\documentclass{amsart} 

\usepackage[english]{babel}
\usepackage[latin1]{inputenc} %Accenti per Windows
\usepackage{amsmath}
\usepackage{amsthm}
\usepackage{amssymb}
\usepackage{tikz}

\usepackage{imakeidx}
\usepackage{appendix}
\usepackage{tikz-cd}
\usepackage{verbatim}
\usepackage{enumitem}
\usepackage{multicol}

\usepackage{hyperref}
\usepackage{cleveref}

\usepackage{mathtools}

\numberwithin{equation}{section} 
\newtheorem{thm}[equation]{Theorem} %Theorem o Teorema a seconda
\newtheorem{question}[equation]{Question}
\newtheorem{prop}[equation]{Proposition}
\newtheorem{lemma}[equation]{Lemma}
\newtheorem{cor}[equation]{Corollary}

\theoremstyle{definition}

\theoremstyle{remark}
\newtheorem{rmk}[equation]{Remark}

\newcommand{\Q}{\mathbb Q}

\newcommand{\Z}{\mathbb Z}
\renewcommand{\L}{\mathbb L}

\renewcommand{\P}{\mathbb P}
\newcommand{\C}{\mathbb C}
\renewcommand{\c}{\subseteq}
\newcommand{\A}{\mathbb A}

\newcommand{\mc}[1]{\mathcal{#1}}

\newcommand{\set}[1]{\{#1\}}
\newcommand{\on}[1]{\operatorname{#1}}

\title[Motivic classes and unramified cohomology]{Motivic classes of classifying stacks of finite groups and unramified cohomology}
\author{Federico Scavia}

\begin{document}

\begin{abstract}
	Combining work of Peyre, Colliot-Th\'el\`ene and Voisin, we give the first example of a finite group $G$ such that the motivic class of its classifying stack $BG$ in Ekedahl's Grothendieck ring of stacks over $\C$ is non-trivial and $BG$ has trivial unramified Brauer group.
\end{abstract}

\maketitle

\section{Introduction}
Let $G$ be a finite group, let $k$ be a field, and let $V$ be a faithful $G$-representation over $k$. 
We say that the classifying stack $BG$ is stably $k$-rational if the quotient $V/G$ is stably $k$-rational, that is, $V/G\times_k \A_k^m$ is birationally equivalent to $\A^n_k$ for some $m,n\geq 0$. By the no-name lemma, this definition does not depend on $V$. The question of the stable rationality of $BG$ is a variation of the following problem, first considered by E. Noether \cite{noether1917gleichungen}: if $V$ is the regular representation of $G$ over $k$, is the field of invariants $k(V)^G$ purely transcendental over $k$? %If $k=\Q$, an affirmative answer to this question implies that the inverse Galois problem for $G$ over $\Q$ has an affirmative answer.
If Noether's problem for $G$ over $k$ has an affirmative answer, then $BG$ is stably $k$-rational.

In \cite{swan1969invariant} and \cite{voskresenskii1970birational}, R. Swan and V. Voskresenski\u{\i} independently constructed the first example of $G$ such that $BG$ is not stably rational. In their example, $k=\Q$ and $G=\Z/47\Z$. Later, D. Saltman remarked that Wang's counterexamples to the Grunwald problem imply that $B(\Z/8\Z)$ is not stably rational over $\Q$. A complete solution to Noether's problem for abelian groups was given by H. Lenstra \cite{lenstra1974rational}.

The first examples over an algebraically closed field were given by Saltman in \cite{saltman1984noether}. Saltman observed that the unramified Brauer group $\on{Br}_{\on{nr}}(K/k)$ of a purely transcendental field extension $K/k$ is trivial, and then construced a finite group $G$ and a $k$-representation $V$ of $G$ such that $\on{Br}_{\on{nr}}(k(V)^G/k)\neq 0$. When $k=\C$, E. Peyre exhibited the first examples of groups $G$ such that $BG$ is not stably rational over $\C$, but $\on{Br}_{\on{nr}}(\C(V)^G/\C)=0$. These examples satisfy $H^3_{\on{nr}}(\C(V)^G/\C,\Q/\Z)\neq 0$; see \cite[Theorem 3.1]{peyre2009unramified}. 

In this paper, we consider a motivic variant of the problem of stable rationality of $BG$. We denote by $K_0(\on{Stacks}_k)$ the Grothendieck ring of algebraic $k$-stacks, as defined by T. Ekedahl in \cite{ekedahl2009grothendieck}; see \Cref{prelim}. By definition, every algebraic stack $\mc{X}$ of finite type over $k$ and with affine stabilizers has a class $\set{\mc{X}}$ in $K_0(\on{Stacks}_k)$. The multiplicative identity of $K_0(\on{Stacks}_k)$ is $1=\set{\on{Spec}k}$.

It is an interesting problem to compute the class $\set{BG}$ in $K_0(\on{Stacks}_{k})$. We have $\set{BG}=1$ in many cases, e.g. when $G=\mu_n$, $G=S_n$, or $G$ is a finite subgroup of $\on{GL}_3$ and $k$ is algebraically closed of characteristic zero; see \cite[Proposition 3.2, Theorem 4.1]{ekedahl2009geometric} and \cite[Theorem 2.4]{martino2016ekedahl}. In all examples of finite groups $G$ such that $\set{BG}=1$, the classifying stack $BG$ is known to be stably rational.

There are also examples of finite groups $G$ for which $\set{BG}\neq 1$ in $K_0(\on{Stacks}_k)$. In \cite[Corollary 5.8]{ekedahl2009geometric}, Ekedahl showed that $\set{B(\Z/47\Z)}\neq 1$ in $K_0(\on{Stacks}_{\Q})$. Moreover, in \cite[Theorem 5.1]{ekedahl2009geometric}, he showed that if $\on{Br}_{\on{nr}}(\C(V)^G/\C)\neq 0$, then $\set{BG}\neq 1$. Thus, the examples $G$ of Swan, Voskresenski\u{\i} and Saltman also satisfy $\set{BG}\neq 1$. It is natural to wonder whether Peyre's examples also satisfy $\set{BG}\neq 1$. 
%\Cref{beyond-br} gives further evidence for an affirmative answer to Totaro's question.

\begin{question}\label{beyond-question}
	Does there exist a finite group $G$ such that $\on{Br}_{\on{nr}}(\C(V)^G/\C)=0$, but $\set{BG}\neq 1$ in $K_0(\on{Stacks}_{\C})$? 
\end{question}
To our knowledge, this question was first asked by Ekedahl, and was posed to us by A. Vistoli.  

If $H^i_{\on{nr}}(\C(V)^G/\C,\Q/\Z)\neq 0$ for some $i$, then $BG$ is not stably rational; see \cite[Proposition 3.4]{merkurjev2017invariants}. When $i\geq 3$, it is not known whether $H^i_{\on{nr}}(\C(V)^G/\C,\Q/\Z)\neq 0$ implies $\set{BG}\neq 1$ in $K_0(\on{Stacks}_{\C})$. We prove that this is the case if $i=3$.

\begin{thm}\label{beyond-br}
	Let $k$ be a field of characteristic zero, let $G$ be a finite group, and let $V$ be a faithful complex representation of $G$. Assume that $H^3_{\on{nr}}(\C(V)^G/\C,\Q/\Z)\neq 0$. Then $\set{BG}\neq 1$ in $K_0(\on{Stacks}_{k})$.
\end{thm}
A crucial ingredient in our proof of \Cref{beyond-br} is a result of J.-L. Colliot-Th\'el\`ene and C. Voisin \cite{colliot2012cohomologie}; see \Cref{ct-voisin} below. %If $V$ is a faithful complex representation of $G$, their result allows us to link $H^3_{\on{nr}}(\C(V)^G/\C,\Q/\Z)$ to the failure of the integral Tate conjecture for codimension $2$ cycles on smooth projective models of $V/G$.

%By \cite{totaro2016motive}, if $H^3_{\on{nr}}(\C(V)^G,\Q/\Z)\neq 0$ then $BG$ is not mixed Tate. Therefore, \Cref{beyond-br} is consistent with the equivalence between $\set{BG}=1$ and the mixed Tate property for $BG$.

The combination of Peyre's examples in \cite{peyre2009unramified} and \Cref{beyond-br} has the following consequence.

\begin{cor}
	\Cref{beyond-question} has an affirmative answer.
\end{cor}

In \cite{totaro2016motive} B. Totaro asked, among other things, whether the stable rationality of $BG$ over $\C$ is equivalent to the condition $\set{BG}=1$ in $K_0(\on{Stacks}_{\C})$. An affirmative answer to Totaro's question is supported by all known examples, and also by \Cref{beyond-br}. A proof of the equivalence seems to be out of reach of current techniques.

%We denote by $K_0(\on{Ab})$ the group generated by isomorphism classes $[A]$ of finitely generated abelian groups $A$, modulo the relations $[A\oplus B]=[A]+[B]$. 
%For every $i\in \Z$ and every smooth projective variety $X$ over $\C$, we denote by $Z_{2i}(X)$ the obstruction to the integral Hodge conjecture for cycles of dimension $i$ on $X$; see \cite{}. 
%We briefly discuss the proof of \Cref{beyond-br}. Let $V$ be a faithful complex representation of $G$, and let $X$ be a smooth projective variety such that $\C(X)\cong \C(V)^G$. In \Cref{welldef}, we construct certain homomorphisms $Z_{2i}:K_0(\on{Stacks}_{\C})\to K_0(\on{Ab})$, and in \Cref{h3} we show that $Z_{-4}(\set{BG})$

\section{The Grothendieck ring of stacks}\label{prelim}
Let $k$ be an arbitrary field. By definition, the Grothendieck ring of varieties $K_0(\on{Var}_{k})$ is the abelian group generated by isomorphism classes $\set{X}$ of $k$-schemes $X$ of finite type, modulo the relations $\set{X}=\set{Y}+\set{X\setminus Y}$ for every closed subscheme $Y\c X$. The multiplication in $K_0(\on{Var}_k)$ is defined on generators by $\set{X}\cdot\set{Y}:=\set{X\times_k Y}$, and we have $1=\set{\on{Spec}k}$. We set $\L:=\set{\A_k^1}$. %If $k$ is of characteristic zero, then $K_0(\on{Var}_k)$ admits the following presentation, due to F. Bittner.

Following Ekedahl \cite{ekedahl2009grothendieck}, we define the Grothendieck ring of stacks $K_0(\on{Stacks}_k)$ as the abelian group generated by isomorphism classes $\set{\mc{X}}$ of algebraic stacks $\mc{X}$ with affine stabilizers and of finite type over $k$, modulo the relations $\set{\mc{X}}=\set{\mc{Y}}+\set{\mc{X}\setminus \mc{Y}}$ for every closed embedding $\mc{Y}\c \mc{X}$, and the relations $\set{\mc{E}}=\set{\A_k^r\times_k \mc{X}}$ for every vector bundle $\mc{E}\to\mc{X}$ of constant rank $r$. The product is defined on generators by $\set{\mc{X}}\cdot\set{\mc{Y}}:=\set{\mc{X}\times_k\mc{Y}}$, and we have $1=\set{\on{Spec}k}$. By \cite[Theorem 1.2]{ekedahl2009grothendieck}, the canonical ring homomorphism $K_0(\on{Var}_k)\to K_0(\on{Stacks}_k)$ induces an isomorphism
\[K_0(\on{Stacks}_k)\cong K_0(\on{Var}_k)[\set{\L^{-1},(\L^n-1)^{-1}:n\geq 1}].\]

The following was observed by Ekedahl in \cite[p. 14]{ekedahl2009grothendieck}.
\begin{lemma}\label{bittner-like}
	Let $k$ be a field of characteristic zero. Then, as an abelian group, $K_0(\on{Var}_k)[\L^{-1}]$ may be presented as the abelian group generated by formal fractions of the form $\set{X}/\L^m$, where $X$ is a smooth projective variety and $m\geq 0$, modulo the following relations:
	\begin{enumerate}[label=(\roman*)]
		\item $\set{\emptyset}=0$,
		\item $\set{\widetilde{X}}/\L^m-\set{X}/\L^m=\set{E}/\L^m-\set{Y}/\L^m$, for every smooth projective variety $X$, every blow-up $\widetilde{X}\to X$ at a smooth closed subscheme $Y\c X$, with exceptional divisor $E\to Y$, and every $m\geq 0$,
		\item $\set{X\times_k \P^1_k}/\L^{m+1}-\set{X}/\L^{m+1}=\set{X}/\L^{m}$, for every smooth projective variety $X$ and every $m\geq 0$. 
	\end{enumerate} 
\end{lemma}

\begin{proof}
	This easily follows from Bittner's presentation of $K_0(\on{Var}_k)$, given in \cite[Theorem 3.1]{bittner2004universal}.
\end{proof}

The dimension filtration $\on{Fil}^{\bullet}K_0(\on{Var}_k)[\L^{-1}]$ of $K_0(\on{Var}_k)[\L^{-1}]$ is defined as follows: for every $n\in \Z$, $\on{Fil}^{n}K_0(\on{Var}_k)[\L^{-1}]$ is the subgroup generated by the elements $\set{X}/\L^m$, where $X$ is a $k$-variety and $\on{dim}(X)-m\leq n$. When $\on{char}k=0$, using resolution of singularities, we see that $\on{Fil}^{n}K_0(\on{Var}_k)[\L^{-1}]$ is generated by elements of the form $\set{X}/\L^m$, where $X$ is a smooth projective $k$-variety and $\on{dim}(X)-m\leq n$; see \cite[Lemma 3.1]{ekedahl2009grothendieck}.

We denote by $\hat{K}_0(\on{Var}_k)$ the completion of $K_0(\on{Var}_k)$ with respect to the dimension filtration. For every $n,n'\in \Z$, we have \[\on{Fil}^{n}K_0(\on{Var}_k)[\L^{-1}]\cdot \on{Fil}^{n'}K_0(\on{Var}_k)[\L^{-1}]\c \on{Fil}^{n+n'}K_0(\on{Var}_k)[\L^{-1}].\] It follows that the multiplication on $K_0(\on{Var}_k)$ extends to $\hat{K}_0(\on{Var}_k)$, making it into a commutative ring with identity. 

For every $n\geq 1$, we have $(1-\L^n)\sum_{i\geq 0} \L^{ni}=1$ in $\hat{K}_0(\on{Var}_k)$. Therefore, we have canonical ring homomorphisms
\[K_0(\on{Var}_k)\to K_0(\on{Stacks}_k)\to \hat{K}_0(\on{Var}_k).\]

\section{The integral Hodge Question and unramified cohomology}
Let $X$ be a smooth projective variety over $\C$, and let $d:=\on{dim}(X)$. For every integer $i$, we write $CH^i(X)$ for the group of algebraic cycles of codimension $i$ on $X$ modulo rational equivalence, and we set $CH_i(X):=CH^{d-i}(X)$.  We have the cycle class maps \[\on{cl}_X^i:CH^i(X)\to H^{2i}(X(\C),\Z).\] %and, using Poincare duality, \[\on{cl}_{X,i}:CH_i(X)\to H_{2i}(X(\C),\Z).\] We denote $H^{2i}_{\on{alg}}(X(\C),\Z):=\on{Im}(\on{cl}_X^i)$. 
By convention, we set $CH^i(X)=0$ and $H^{2i}(X(\C),\Z)=0$ when $i<0$ and $i>d$.

A cohomology class $\alpha \in H^{2i}(X(\C),\Z)$ is called an integral Hodge class if its image in $H^{2i}(X(\C),\C)$ is of type $(i,i)$ with respect to the Hodge decomposition of $H^{2i}(X(\C),\C)$. We denote by $\on{Hdg}^{2i}(X,\Z)$ the subgroup of integral Hodge classes of $H^{2i}(X(\C),\Z)$. We have an inclusion $\on{Im}(\on{cl}_X^i)\c \on{Hdg}^{2i}(X,\Z)$. We set \[Z^{2i}(X):=\on{Hdg}^{2i}(X,\Z)/\on{Im}(\on{cl}_X^i),\qquad Z_{2i}(X):=Z^{2d-2i}(X).\] %. Using Poincar\'e duality, we may similarly define the groups $Z_{2i}(X)$, 

For every integer $i$, the abelian group $Z^{2i}(X)$ is finitely generated. The Hodge Conjecture for cycles of codimension $i$ on $X$ predicts that $Z^{2i}(X)$ is finite. The integral Hodge Question for cycles of codimension $i$ on $X$ asks whether $Z^{2i}(X)$ is zero. By the Lefschetz Theorem on $(1,1)$-classes, the integral Hodge Question has an affirmative answer when $i=1$. When $i=2$, the integral Hodge Question has a negative answer in general, as shown by examples of M. Atiyah and F. Hirzebruch \cite{atiyah1962analytic}.

\begin{thm}[Colliot-Th\'el\`ene, Voisin]\label{ct-voisin}
	Let $X$ be a smooth projective variety over $\C$, of dimension $d$. Assume that there exists a smooth closed subvariety $S\c X$ of dimension $\leq 2$, such that the pushforward map $CH_0(S)\to CH_0(X)$ is surjective. Then we have an isomorphism of finite groups \[H^3_{\on{nr}}(X,\Q/\Z)\cong Z^4(X).\]
\end{thm}

\begin{proof}
	See \cite[Th\'eor\`eme 1.1]{colliot2012cohomologie}.
\end{proof}

%Each group appearing in the short exact sequence of \Cref{ct-voisin}(a) is a birational invariant of the smooth projective variety $X$; we refer the reader to \cite{colliot2012cohomologie} for the definitions and properties of these groups.
\begin{rmk}\label{unirat}
	It is well known that the assumptions of \Cref{ct-voisin} are satisfied when $X$ is unirational. We have learned the following argument from J.-L. Colliot-Th\'el\`ene.
	
	If $X$ is a smooth projective unirational variety over $\C$, then there exist a dense open subset $U\c X$ and a surjective morphism $\varphi:V\to U$, where $V$ is an open subset of some affine space. If $p_1,p_2\in U(\C)$, we may find $q_1,q_2\in V(\C)$ such that $\varphi(q_i)=p_i$ for $i=1,2$. There is a line connecting $q_1$ and $q_2$, hence, since $X$ is complete, we find a morphism $\P^1\to X$ whose image contains $p_1$ and $p_2$. It follows that any two zero-cyles of degree $1$ in $U$ are rationally equivalent.
	
	Now, if $p\in X(\C)$, a moving lemma shows that $p$ is rationally equivalent to a zero-cycle whose support is contained in $U$; see \cite[Compl\'ement, p. 599]{colliot2005theoreme}. We conclude that the degree map $\on{deg}:CH_0(X)\to \Z$ is an isomorphism. Thus, the hypotheses of \Cref{ct-voisin} are satisfied, with $S$ a closed point of $X$.
	
	%\Cref{ct-voisin} applies in particular when $X$ is unirational. Indeed, if $X$ is unirational, then $X$ is rationally connected, that is, any two general points of $X$ can be connected by a rational curve; see \cite[Remark 4.4(1)]{debarre2001higher}. Since $\C$ is uncountable, this implies that any two points of $X$ can be connected by a chain of rational curves; see \cite[Remark 4.22(2)]{debarre2001higher}. In particular, the degree map $\on{deg}:CH_0(X)\to \Z$ is an isomorphism. Thus, the hypotheses of \Cref{ct-voisin} are satisfied, with $S$ a closed point of $X$.
\end{rmk}

\section{Proof of Theorem \ref{beyond-br}}\label{beyond-sec}
We denote by $K_0(\on{Ab})$ the group generated by isomorphism classes $[A]$ of finitely generated abelian groups $A$, modulo the relations $[A\oplus B]=[A]+[B]$. As an abelian group, $K_0(\on{Ab})$ is freely generated by $[\Z]$ and $[\Z/p^n\Z]$, where $p$ ranges among prime numbers and $n\geq 1$; see \cite[Proposition 3.3(i)]{ekedahl2009grothendieck}.

%We now use the groups $Z_{2i}(X)$ to construct some group homomorphisms with domain $K_0(\on{Var}_{\C})[\L^{-1}]$.

\begin{prop}\label{welldef} 
	Let $i$ be an integer.
	\begin{enumerate}[label=(\alph*)]
		\item There exists a group homomorphism \[Z_{2i}: K_0(\on{Var}_{\C})[\L^{-1}]\to K_0(\on{Ab}),\] given by sending $\set{X}/\L^m\mapsto [Z_{2i+2m}(X)]$ for every smooth projective variety $X$ over $\C$ and every $m\geq 0$.
		\item The homomorphism $Z_{2i}$ is continuous with respect to the filtration topology on $K_0(\on{Var}_{\C})[\L^{-1}]$ and the discrete topology of $K_0(\on{Ab})$. It thus extends uniquely to a group homomorphism \[\hat{Z}_{2i}: \hat{K}_0(\on{Var}_{\C})\to K_0(\on{Ab}).\]
	\end{enumerate} 
\end{prop}

\begin{proof}
	(a) To show that $Z_{2i}: K_0(\on{Var}_{\C})[\L^{-1}]\to K_0(\on{Ab})$ is well-defined, we verify that the association $\set{X}/\L^m\mapsto [Z_{2i+2m}(X)]$ respects the relations of \Cref{bittner-like}. It is clear that (i) is satisfied.
	
	Let $m\geq 0$, let $Y\c X$ is a closed embedding of smooth projective complex varieties, let $\widetilde{X}\to X$ be the blow-up of $X$ at $Y$, and let $E$ be the exceptional divisor of the blow-up. Denote by $d$ the dimension of $X$, and by $r$ the codimension of $Y$ in $X$. We want to show that
	\begin{equation}\label{ct-idea}[Z_{2i+2m}(X)]-[Z_{2i+2m}(Y)]=[Z_{2i+2m}(\widetilde{X})]-[Z_{2i+2m}(E)].\end{equation}
	in $K_0(\on{Ab})$. Letting $j=d-i-m$, we see that (\ref{ct-idea}) is equivalent to:
	\begin{equation}\label{bittner-z}[Z^{2j}(X)]-[Z^{2j-2r}(Y)]=[Z^{2j}(\widetilde{X})]-[Z^{2j-2}(E)].\end{equation}
	By \cite[Theorem 9.27]{voisin2007hodge2}, we have a group isomorphism 
	\[\varphi^j:\oplus_{0\leq h\leq r-2}CH^{j-1-h}(Y)\oplus CH^j(X)\xrightarrow{\sim} CH^j(\widetilde{X}).\]
	By \cite[Theorem 7.31]{voisin2007hodge1}, we have an isomorphism of Hodge structures	
	\[\oplus_{0\leq h\leq r-2}H^{2j-2-2h}(Y(\C),\Z)\oplus H^{2j}(X(\C),\Z) \xrightarrow{\sim} H^{2j}(\widetilde{X}(\C),\Z),\] where the Hodge structure on $H^{2j-2-2h}(Y(\C),\Z)$ is shifted by $(h+1,h+1)$, and so has weight $2j$. In particular, we have an isomorphism of groups
	\[\psi^j:\oplus_{0\leq h\leq r-2}\on{Hdg}^{2j-2-2h}(Y,\Z)\oplus \on{Hdg}^{2j}(X,\Z)\xrightarrow{\sim} \on{Hdg}^{2j}(\widetilde{X},\Z).\]
	Comparing the explicit description of these isomorphisms, as given in the references, we see that $\varphi^j$ and $\psi^j$ are compatible with the cycle class maps. In other words, we have a commutative square	
	\[
	\begin{tikzcd}
	\oplus_{0\leq h\leq r-2}CH^{j-1-h}(Y)\oplus CH^j(X)\arrow[r,"\varphi^j"] \arrow[d,"\oplus_h\on{cl}_Y^h\oplus \on{cl}_X^j"] & CH^j(\widetilde{X}) \arrow[d,"\on{cl}_{\widetilde{X}}^j"] \\
	\oplus_{0\leq h\leq r-2}\on{Hdg}^{2j-2-2h}(Y,\Z)\oplus \on{Hdg}^{2j}(X,\Z) \arrow[r,"\psi^j"] & \on{Hdg}^{2j}(\widetilde{X},\Z).		
	\end{tikzcd}
	\]
	We deduce that \begin{equation}\label{blowup}Z^{2j}(\widetilde{X})\cong \oplus_{0\leq h\leq r-2} Z^{2j-2-2h}(Y)\oplus Z^{2j}(X).\end{equation}
	The morphism $E\to Y$ identifies $E$ with the projectivization of the normal bundle of $Y$ inside $X$. By \cite[Theorem 9.25]{voisin2007hodge2}\footnote{Note that the formula of \cite[Theorem 9.25]{voisin2007hodge2} contains a typographical error: $CH_{l-r-1+k}(-)$ should be $CH_{l-r+1+k}(-)$.} and \cite[Lemma 7.32]{voisin2007hodge1}, the pullback along $E\to Y$ induces a commutative diagram
	\[
	\begin{tikzcd}
	\oplus_{0\leq h\leq r-1}CH^{j-1-h}(Y)\arrow[r,"\sim"] \arrow[d,"\oplus_h\on{cl}_Y^h"] & CH^{j-1}(E) \arrow[d,"\on{cl}_E^{j-1}"] \\
	\oplus_{0\leq h\leq r-1}\on{Hdg}^{2j-2-2h}(Y,\Z) \arrow[r,"\sim"] & \on{Hdg}^{2j-2}(E,\Z),		
	\end{tikzcd}
	\]
	where the horizontal arrows are isomorphisms.
	We deduce that \begin{equation}\label{projbundle}Z^{2j-2}(E)\cong \oplus_{0\leq h\leq r-1} Z^{2j-2-2h}(Y).\end{equation}
	Now (\ref{bittner-z}) follows from (\ref{blowup}) and (\ref{projbundle}). Therefore, $Z_{2i}$ respects all relations of type (ii).
	
	It remains to show that $Z_{2i}$ is compatible with relations of type (iii). Let $X$ be a smooth projective variety of dimension $d$, and let $m\geq 0$ be an integer. We must show that
	\[[Z_{2i+2m+2}(X\times_{\C}\P^1_{\C})]-[Z_{2i+2m+2}(X)]=[Z_{2i+2m}(X)].\]
	Setting $j=d-i-m$, the claim becomes
	\begin{equation}\label{PisL-1}
	[Z^{2j}(X\times_{\C}\P^1_{\C})]-[Z^{2j-2}(X)]=[Z^{2j}(X)].
	\end{equation}
	Applying \cite[Theorem 9.25]{voisin2007hodge2} and \cite[Lemma 7.32]{voisin2007hodge1} to the trivial projective bundle $X\times_{\C}\P^1_{\C}\to X$, we obtain a commutative square 
	\[
	\begin{tikzcd}
	CH^{j}(X)\oplus CH^{j-1}(X)\arrow[r,"\sim"] \arrow[d,"\on{cl}_X^j\oplus \on{cl}_X^{j-1}"] & CH^{j}(X\times_{\C}\P^1_{\C}) \arrow[d,"\on{cl}_{X\times_{\C} \P^1_{\C}}^{j}"] \\
	\on{Hdg}^{2j}(X,\Z)\oplus \on{Hdg}^{2j-2}(X,\Z) \arrow[r,"\sim"] & \on{Hdg}^{2j}(X\times_{\C}\P^1_{\C},\Z).		
	\end{tikzcd}
	\]
	Thus \[Z^{2j}(X\times_{\C}\P^1_{\C})\cong Z^{2j}(X)\oplus Z^{2j-2}(X),\] which implies (\ref{PisL-1}). It follows that $Z_{2i}$ respects relations of type (iii) as well, hence $Z_{2i}$ is a well-defined group homomorphism.
	
	(b) Let $X$ be a smooth projective variety of dimension $d$, and let $m\geq d-i$. Then $2i+2m\geq 2d$, and so \[Z_{2i}(\set{X}/\L^m)=[Z_{2i+2m}(X)]=0.\] This means that $Z_{2i}$ sends $\on{Fil}^iK_0(\on{Var}_{\C})[\L^{-1}]$ to zero. Therefore, if we endow $K_0(\on{Var}_{\C})[\L^{-1}]$ with the dimension filtration topology and $K_0(\on{Ab})$ with the discrete topology, the homomorphism $Z_{2i}$ is continuous. It follows that $Z_{2i}$ extends uniquely to a homomorphism $\hat{Z}_{2i}:\hat{K}_0(\on{Var}_{\C})\to K_0(\on{Ab})$.
\end{proof}
We also denote by $Z_{2i}$ the composition \[K_0(\on{Stacks}_{\C})\to \hat{K}_0(\on{Var}_{\C})\xrightarrow{\hat{Z}_{2i}} K_0(\on{Ab}).\]

%Let $G$ be a finite group, let $V$ be a faithful $G$-representation over $\C$, and let $X$ be a smooth projective variety that is birational to $V/G$. %By \cite[Theorem 5.1]{ekedahl2009geometric}, $H^{i}(\set{BG})=0$ for $i>0$ and $i=-1$, $H^0(\set{BG})=[\Z]$ and $H^{-2}(\set{BG})=[\on{Br}(\C(X))]$. We now consider the situation for $Z_{2i}$.

\begin{prop}\label{h3}
	Let $G$ be a finite group, let $V$ be a faithful $G$-representation over $\C$, and let $X$ be a smooth projective variety over $\C$ that is birational to $V/G$. Then $Z_{2i}(\set{BG})=0$ for every $i\geq -1$, and
	\[Z_{-4}(\set{BG})=[H^3_{\on{nr}}(\C(V)^G/\C,\Q/\Z)].\]
\end{prop}

\begin{proof}
	Let $V$ be a faithful complex $G$-representation of dimension $d\geq 1$, and let $U\c V$ be the non-empty open subscheme where $G$ acts freely. By \cite[Theorem 3.4]{ekedahl2009geometric}, we may write
	\begin{equation}\label{h3-eq1}\set{BG}\L^{d}=\set{U/G}+\sum_j m_j\set{BH_j}\L^{a_j}\end{equation} in $K_0(\on{Stacks}_{\C})$, where the $H_j$ are distinct proper subgroups of $G$, $m_j\in\Z$ and $a_j\leq d-1$.
	
	Using resolution of singularities, we may write \begin{equation}\label{h3-eq2}\set{U/G}=\set{X}+\sum_q n_q\set{X_q}\end{equation} in $K_0(\on{Var}_{\C})$, where $X$ and the $X_q$ are smooth projective varieties over $\C$, $X$ is birationally equivalent to $U/G$, $\on{dim}(X_q)\leq d-1$, and $n_q\in \Z$ for every $q$. We substitute (\ref{h3-eq2}) into (\ref{h3-eq1}) and divide by $\L^d$: 
	\[\set{BG}=\set{X}\L^{-d}+\sum_q n_q\set{X_q}\L^{-d}+\sum_j m_j\set{BH_j}\L^{a_j-d}.\]
	We apply $Z_{2i}$:
	\begin{equation}\label{hodge-induct}Z_{2i}(\set{BG})=[Z_{2i+2d}({X})]+\sum_q n_q[Z_{2i+2d}({X_q})]+\sum_j m_jZ_{2i+2d-2a_j}(\set{BH_j}).\end{equation}
	If $G$ is trivial, then $BG\cong \on{Spec}{\C}$ and there is nothing to prove. Assume now that $G$ is non-trivial, and that the conclusion of the proposition holds for all $i\in \Z$ and all proper subgroups of $G$. 
	
	By the Lefschetz Theorem on $(1,1)$-classes, we have $Z_{2d-2}(X)=Z^2(X)=0$. Therefore, if $i\geq -1$ every term on the right hand side of (\ref{hodge-induct}) is zero. This shows that $Z^{i}(\set{BG})=0$ for all $i\geq -1$.
	
	If $i=-2$, another application of the Lefschetz Theorem on $(1,1)$-classes shows that the right hand side of (\ref{hodge-induct}) reduces to $[Z_{2d-4}(X)]=[Z^4(X)]$. Since $X$ is birationally equivalent to $V/G$, by \Cref{ct-voisin} we have: \[Z^4(X)\cong H^3_{\on{nr}}(X,\Q/\Z)\cong H^3_{\on{nr}}(\C(V)^G/	\C,\Q/\Z).\qedhere\] 
\end{proof}

\begin{proof}[Proof of \Cref{beyond-br}]
	By a standard limit argument, we may assume that $k$ is finitely generated over $\Q$. Fix an embedding $k\hookrightarrow \C$. By assumption, we have $H^3_{\on{nr}}(\C(V)^G/\C,\Q/\Z)\neq 0$, hence by \Cref{h3} we obtain $Z_{-4}(\set{BG})\neq 0$. On the other hand, it is clear that $Z_{-4}(\set{\on{Spec}\C})=0$. We conclude that $\set{BG}\neq 1$ in $K_0(\on{Stacks}_{\C})$, hence $\set{BG}\neq 1$ in $K_0(\on{Stacks}_k)$.
\end{proof}

\section*{Acknowledgements}
I thank  Jean-Louis Colliot-Th\'el\`ene for making me aware of the results of \cite{colliot2012cohomologie}, and for a conversation which eventually led me to \Cref{beyond-br}. I thank Angelo Vistoli for posing \Cref{beyond-question} to me, and my advisor Zinovy Reichstein for helpful suggestions on the exposition.

\end{document}